\documentclass{amsart}
\usepackage{enumerate}
\usepackage{amsmath,amssymb,amsfonts}
\usepackage[all]{xy}
\usepackage{tikz-cd}
\usepackage{bm}
\usepackage[total={6.5in,8.75in},
top=1.2in, left=0.9in, includefoot]{geometry}
\usepackage{parskip}

\newtheorem{theorem}{Theorem}
\newtheorem{lemma}[theorem]{Lemma}
\theoremstyle{definition}
\newtheorem{definition}[theorem]{Definition}
\newtheorem{example}[theorem]{Example}
\newtheorem{remark}[theorem]{Remark}

\begin{document}
		
	\title{An effective proof of the Cartan formula: the even prime}
	
	\author{Anibal M. Medina-Mardones} % [label1, label2]
	
	\email{anibal.medinamardones@epfl.ch}
%	\ead{anibal.medinamardones@epfl.ch}
%	\ead[url]{medina-mardones.com}
	
	\address{Laboratory for Topology and Neuroscience, École Polytechnique Fédérale de Lausanne, Switzerland}
	\address{Department of Mathematics, University of Notre Dame du Lac, Notre Dame, IN, USA}
	
	\maketitle
	
	\begin{abstract}
		The Cartan formula encodes the relationship between the cup product and the action of the Steenrod algebra in $\mathbb F_p$-cohomology. In this work, we present an effective proof of the Cartan formula at the cochain level when the field is $\mathbb F_2$. More explicitly, for an arbitrary pair of cocycles and any non-negative integer, we construct a natural coboundary that descends to the associated instance of the Cartan formula. Our construction works for general algebras over the Barratt-Eccles operad, in particular, for the singular cochains of spaces.
	\end{abstract}
	
%	\begin{keyword}
%		%% keywords here, in the form: keyword \sep keyword
%		Cohomology operations \sep Steenrod squares \sep Cartan relation \sep cup product \sep cohomology algebra \sep operads  
%		%% PACS codes here, in the form: \PACS code \sep code
%		
%		%% MSC codes here, in the form: \MSC code \sep code
%		%% or \MSC[2008] code \sep code (2000 is the default)
%		\MSC 55S10 \sep  55S05 \sep 55S12			
%	\end{keyword}
	
	\section{Introduction}
	
	Let $X$ be a space. In \cite{steenrod47products}, Steenrod introduced formulae to define his famous \textbf{Steenrod squares}
	\begin{equation*}
	Sq^k : H^*(X; \mathbb F_2) \to H^*(X; \mathbb F_2)
	\end{equation*}
	and in \cite{steenrod62operations}, he axiomatically characterized them by the following:
	\begin{enumerate} [1.]
		\item $Sq^k$ is natural,
		\item $Sq^0$ is the identity,
		\item $Sq^k(x) = x^2$ for $x \in H^{-k}(X; \mathbb F_2)$,
		\item $Sq^k(x) = 0$ for $x \in H^{-n}(X; \mathbb F_2)$ with $n>k$,
		\item $Sq^k(xy) = \sum_{i+j=k} Sq^i (x) Sq^j(y)$.
	\end{enumerate}
	Axiom 5., known as the \textbf{Cartan formula}, is the focus of this work. 
	
	To describe our viewpoint and present the contributions of this paper in context, let us revisit some of the history of Steenrod's construction.
	
	In the late thirties, Alexander, Whitney, and \v{C}ech defined the ring structure on cohomology 
	\begin{equation} \label{equation: definition cup product}
	[\alpha] [\beta] = [\alpha \smallsmile_{0} \beta]
	\end{equation}
	using a cochain level construction
	\begin{equation*}
	\smallsmile_0 \, : N^*(X; \mathbb Z) \otimes N^*(X; \mathbb Z) \to N^*(X; \mathbb Z)
	\end{equation*}
	dual to a choice of simplicial chain approximation to the diagonal inclusion. 
	
	Steenrod then showed that $\smallsmile_0$ is commutative up to coherent homotopies by effectively constructing \textbf{cup-$i$ products}
	\begin{equation*}
	\smallsmile_i\, : N^*(X; \mathbb Z) \otimes N^*(X; \mathbb Z) \to N^*(X; \mathbb Z)
	\end{equation*}
	enforcing its derived commutativity. (Axioms for these and connections with higher category theory can be found in \cite{medina2018axiomatic} and \cite{medina2019globular}).  Then, with coefficients in $\mathbb F_2$, Steenrod defined
	\begin{equation} \label{equation: definition steenrod squares}
	Sq^k\big([\alpha]\big) = [\alpha \smallsmile_{k-n} \alpha].
	\end{equation}
	This definition of the Steenrod squares makes the Cartan formula equivalent to 
	\begin{equation} \label{equation: Cartan 1}
	0 = 
	\Big[ (\alpha \smallsmile_0 \beta) \smallsmile_i (\alpha \smallsmile_0 \beta)\ + \sum_{i=j+k} (\alpha \smallsmile_j \alpha) \smallsmile_0 (\beta \smallsmile_k \beta) \Big].
	\end{equation}
	The goal of this work is to effectively construct for any $i \geq 0$ and arbitrary pair of cocycles $\alpha, \beta \in N^*(X; \mathbb F_2)$ a natural cochain $\zeta_i(\alpha \otimes \beta)$ such that
	\begin{equation} \label{equation: Cartan 2}
	\delta \zeta_i(\alpha \otimes \beta) = 
	(\alpha \smallsmile_0 \beta) \smallsmile_i (\alpha \smallsmile_0 \beta)\ + \sum_{i=j+k} (\alpha \smallsmile_j \alpha) \smallsmile_0 (\beta \smallsmile_k \beta).
	\end{equation}	
	Following \cite{may70generalapproach}, we take a more general approach and in doing so we describe a non-necessarily effective construction for any \mbox{$E_\infty$-algebra}. We work over a fixed algebraic model $\mathcal E$ of the \mbox{$E_\infty$-operad }known as the Barratt-Eccles operad. This model, introduced by Berger-Fresse in \cite{berger04combinatorial}, is equipped with a diagonal map and the natural $\mathcal E$-algebra structure defined by these authors on the normalized cochains of simplicial sets is suitable for effective constructions.
	
	This paper is part of an ongoing effort spearheaded by Greg Brumfiel and John Morgan to build effective models of classical homotopy-theoretic concepts, see for example \cite{brumfiel2016pontrjagin} and \cite{brumfiel2018quadratic}. Motivation for this especific project came from a question of Anton Kapustin. See \cite{kapustin2017fermionic} for an instance where one of our formulae is used in the context of topological phases of matter.
	
	Implementations of the constructions of this paper and of state-of-the-art algorithms for the computation of Steenrod squares and \mbox{cup-$i$} products, as introduced in \cite{medina2018persistence}, can be found in the author's website.
	
	\subsection*{Acknowledgement} We would like to thank John Morgan, Greg Brumfiel, Dennis Sullivan, Anton Kapustin, Mark Behrens, Marc Stephan, and Kathryn Hess for their insights, questions, and comments about this project.
	
	\section{Conventions and preliminaries}
	
	For the remainder of this paper all algebraic constructions are considered over $\mathbb F_2$, the field with two elements. 
	
	\subsection{Chain complexes and simplicial sets}
	 
	We denote the category of chain complexes of \mbox{$\mathbb F_2$-modules} by $\mathbf{Ch}$. Boundary maps decrease degree and for chain complexes $C$ and $C'$ the set of linear maps $\mathrm{Hom}(C, C')$ is a chain complex with 
	\begin{equation*}
	\mathrm{Hom}(C,C')_n = \{f\ :\ c \in C_m \text{ implies } f(c) \in C'_{m+n}\}
	\qquad
	\partial f = \partial \circ f + f \circ \partial.
	\end{equation*}
	Let $f, g, h \in \mathrm{Hom}(C,C')$. We remark that $f$ is a chain map if and only if it is a degree $0$ cycle. Assume $f$ and $g$ are chain maps, then $h$ is a chain homotopy between them if and only if 
	\begin{equation*}
	\partial h = f + g.
	\end{equation*}
	We remark that $f \in \mathrm{Hom}(C,C')$ is a chain map if and only if it is a degree $0$ cycle, and that $h \in \mathrm{Hom}(C,C')$ is a chain homotopy between two chain maps $f$ and $g$ if and only if 
	\begin{equation*}
	\partial h = f + g.
	\end{equation*}
	
	The product of two chain complexes $C$ and $C'$ is defined by
	\begin{equation*}
	(C \otimes C')_n = \bigoplus_{i+j=n} C_i \otimes C'_j\ \qquad
	\partial = \partial \otimes \mathrm{id} + \mathrm{id} \otimes \partial.
	\end{equation*}
	
	The category $\bm{\Delta}$ is defined to have an object $[n] = \{0, \dots, n\}$ for every non-negative integer $n$ and a morphism $[m] \to [n]$ for each order-preserving function from $\{0, \dots, m\}$ to $\{0, \dots, n\}$.
	
	For integers $0 \leq i \leq n$, the morphisms 
	\begin{equation*}
	\delta_i : [n-1] \to [n] \qquad \sigma_i : [n+1] \to [n]
	\end{equation*}
	defined by
	\begin{equation*}
	\delta_i(k) = 
	\begin{cases} k & k < i \\ k+1 & i \leq k \end{cases}
	\quad \text{ and } \quad
	\sigma_i(k) = 
	\begin{cases} k & k \leq i \\ k-1 & i < k \end{cases}
	\end{equation*} 
	generate all morphisms in $\bm{\Delta}$.
	
	A \textbf{simplicial set} is a contravariant functor from $\bm{\Delta}$ to the category of  sets. We denote the category of simplicial sets by $\mathbf{sSet}$  and use the standard notation
	\begin{equation*}
	X[n] = X_n \qquad X(\delta_i) = d_i \qquad X(\sigma_i) = s_{i.}
	\end{equation*}
	The product of two simplicial sets $X$ and $Y$ is defined by
	\begin{equation*}
	(X \times Y)_n = X_n \times Y_n \qquad
	d_i = d_i \times d_i \qquad
	s_i = s_i \times s_{i.}
	\end{equation*}

	The functor of \textbf{normalized chains}
	\begin{equation*}
	N_* : \mathbf{sSet} \to \mathbf{Ch}
	\end{equation*} 
	is defined by
	\begin{equation*}
	N_n(X) = \frac{\mathbb F_2 \{ X_n \}}{\mathbb F_2 \{ s(X_{n-1}) \}} \ \qquad
	\partial_n = \sum_{i=0}^{n} d_{i}
	\end{equation*}
	where $s(X_{n-1}) = \bigcup_{i=0}^{n-1} s_i(X_{n-1})$.
	
	The functor of \textbf{normalized cochains} $N^*$ is defined by composing $N_*$ with the linear duality functor $\mathrm{Hom}(-, \mathbb F_2) : \mathbf{Ch} \to \mathbf{Ch}$. Notice that in this definition cochains are concentrated in non-positive degrees.
	
	\begin{remark}
		The singular cochains with $\mathbb F_2$-coefficients of a topological space $X$ coincide with $N^*(Sing(X))$, where $Sing(X)$ is the simplicial set of continuous maps from standard topological simplices to $X$.
	\end{remark}
	
	\subsection{The Alexander-Whitney and Eilenberg-Zilber maps}
	
	The functor of normalized chains does not preserve products, i.e., for a general pair of simplicial sets $X$ and $Y$ there is no isomorphism beetween $N_*(X \times Y)$ and $N_*(X) \otimes N_*(Y)$. Nevertheless, there is a canonical chain homotopy equivalence between them.
	
	The \textbf{Alexander-Whitney map}
	\begin{equation*}
	AW: N_*(X \times Y) \to N_*(X) \otimes N_*(Y)
	\end{equation*}
	is defined for $x \times y \in N_n(X \times Y)$ by
	\begin{equation*}
	AW(x \times y) = \sum_{i=0}^n d_{i+1} \cdots d_n\, x \otimes d_0 \cdots d_{i-1}\, y
	\end{equation*}
	and the \textbf{Eilenberg-Zilber map}
	\begin{equation*}
	EZ: N_*(X) \otimes N_*(Y) \to N_*(X \otimes Y)
	\end{equation*}
	is defined for $x \otimes y \in N_p(X) \otimes N_q(Y)$ by
	\begin{equation*}
	EZ(x \otimes y) = 
	\sum s_{v_p} \cdots s_{v_1}\, x \otimes s_{w_q} \cdots s_{w_1}\, y
	\end{equation*}
	where the sum is over all pairs of disjoint subsets
	\begin{equation*}
	\big( \{v_1 < \cdots < v_p\},\ \{w_1 < \cdots < w_q\} \big)
	\end{equation*}
	of $\{0,\dots,p+q-1\}$.
	
	It is well known that the compositions $AWEZ$ is equal to the identity and that the composition $EZAW$ is effectively chain homotopic to the identity. A recursive description for one such chain homotopy was first given in \cite{eilenberg1953groups}. A close formula for it, which we learned from \cite{real2000homological} and is credited therein to \cite{rubio1991homologie}, is given next.
	
	The \textbf{Shih homotopy}
	\begin{equation*}
	SHI: N_n(X \times Y) \to N_*(X \times Y)
	\end{equation*}	 
	is defined for $x \times y \in N_n(X \otimes Y)$ to be $0$ if $n = 0$ and if $n > 0$ by
	\begin{equation*}
	\begin{split}
	SHI(x \times y) = 
	\sum\ & s_{v_p + m} \cdots s_{v_1 + m} s_{m-1} d_{n-p+1} \cdots d_n  \,  x\ \otimes\\ 
	      & s_{w_{q+1} + m} \cdots s_{w_1 + m} d_{n-p-q} \cdots d_{n-p-1} \,  y
	\end{split}
	\end{equation*}
	where $m = n-p-q$ and the sum is over all pairs of disjoint subsets
	\begin{equation*}
	\big( \{v_1 < \cdots < v_{p}\},\ \{w_1 < \cdots < w_{q+1}\} \big)
	\end{equation*}
	of $\{0,\dots,p+q\}$ with $0 \leq p \leq n-1$ and $0 \leq q \leq n-p-1$.
	
	\subsection{Group actions and algebras over operads}
	
	Let $\Gamma$ be a group. A $\Gamma$-action on an object $C$ is a group morphism 
	\begin{equation} \label{equation: group action}
	\Gamma \to \mathrm{Aut}(C)
	\end{equation}
	where the group structure on $\mathrm{Aut}(C)$ is given by composition.
	
	Let $\mathcal{O}$ be an operad. (See for example \cite{loday2012operads}.) An $\mathcal{O}$-algebra structure on an object $C$ is an operad morphism
	\begin{equation} \label{equation: algebra over an operad}
	\mathcal{O} \to \mathrm{End}(C)
	\end{equation} 
	where the operad structure on $\mathrm{End}(C) = \{\mathrm{Hom}(C^{\otimes n}, C)\}_{n \geq 0}$ is induced from composition of linear maps and transpositions of factors.
	
	Let $C$ be an object with a $\Gamma$-action or an $\mathcal{O}$-algebra structure. We identify the elements of $\Gamma$ or $\mathcal{O}$ with their images via $(\ref{equation: group action})$ or $(\ref{equation: algebra over an operad})$.
	
	\subsection{The Barratt-Eccles operad $\mathcal E$}
	
	We review from \cite{berger04combinatorial} the Barratt-Eccles operad. 
	
	For a positive integer $r$ let $\Sigma_r$ be the group of permutations of $r$ elements and
	\begin{equation*}
	\bm{\circ}_{\Sigma} : \Sigma_r \times \Sigma_{s_1} \times \cdots \times \Sigma_{s_r} \ \to \
	\Sigma_{s_1 + \cdots + s_{r}} 
	\end{equation*}
	the usual composition of permutations.
	
	For a positive integer $r$ define the simplicial set $E(r)$ by 
	\begin{equation*}
	\begin{split}
	E(r)_n &= \{ (\sigma_0, \dots, \sigma_n)\ |\ \sigma_i \in \Sigma_r\}, \\
	d_i(\sigma_0, \dots, \sigma_n) &= (\sigma_0, \dots, \widehat{\sigma}_i, \dots, \sigma_n), \\
	s_i(\sigma_0, \dots, \sigma_n) &= (\sigma_0, \dots, \sigma_i, \sigma_i, \dots, \sigma_n). \\
	\end{split}
	\end{equation*}
	We consider $E(r)$ equipped with the action of $\Sigma_r$ given by 
	\begin{equation*}
	\sigma \big((\sigma_0, \dots, \sigma_n)\big) = (\sigma\sigma_0, \dots, \sigma\sigma_n).
	\end{equation*} 
	Let
	\begin{equation*}
	\bm{\circ}_{E}: E(r) \times E(s_1) \times \cdots \times E(s_r) \to E(s_1 + \cdots + s_r)
	\end{equation*}
	be defined by applying $\bm{\circ}_{\Sigma}$ coordinatewise. 
	
	For $r \geq 0$, let
	\begin{equation*}
	\mathcal E(r) = N_*(E(r))
	\end{equation*}
	be edowed with the induced $\Sigma_r$-action. Define an operadic compositions $\bm{\circ}_{\mathcal E}$ on $\mathcal E = \{\mathcal E(r)\}_{r \geq 0}$ by
	\begin{center}
		\begin{tikzcd}
		\mathcal E(r) \otimes \mathcal E(s_1) \otimes \cdots \otimes \mathcal E(s_r)  \ar[r, "EZ^{r}"] \ar[dr, in = 180, out = -40, "\bm{\circ}_{\mathcal E}"]&
		N_* \big( E(r) \otimes E(s_1) \otimes \cdots \otimes E(s_r) \big) \ar[d, "N_*(\bm{\circ}_{E})"] \\ &
		\mathcal E(s_1 + \cdots + s_r)
		\end{tikzcd}
	\end{center}
	where $EZ^{r}$ stand recursively for $EZ(\mathrm{id} \otimes EZ^{r-1})$ with $EZ^1 = EZ$.
	
	The resulting operad $\mathcal E$ is referred to as the Barratt-Eccles operad. It is a model in the category $\mathbf{Ch}$ for the $E_\infty$-operad. That is to say, $E(0)= \mathbb F_2$ and, for $r>0$, $E(r)$ is a resolution of $\mathbb F_2$ by free $\mathbb F_2[\Sigma_r]$-modules.
	
	\section{Cartan coboundaries for $\mathcal E$-algebras}
	
	\subsection{Steenrod cup-$i$ products and Cartan coboundaries}
	
	Since $\mathcal E$ is an $E_\infty$-operad, the orbit complex $\mathcal E(2)_{\Sigma_2}$ is an algebraic model for $K(\mathbb F_2,1)$, and $H_i(\mathcal E(2)_{\Sigma_2})$ is generated as an $\mathbb F_2$-module by the orbit of the element
	\begin{equation*}
	\tilde{x}_i = \big(e, (12), e, \dots, (12)^i \big).
	\end{equation*}
	
	\begin{definition}
		Let $A$ be an $\mathcal{E}$-algebra. The \textbf{cup-$i$ product} of $A$ is the image of $\tilde{x}_i$ in $\mathrm{Hom}(A \otimes A, A)$. We use the notation
		\begin{equation*}
		\tilde{x}_i (\alpha \otimes \beta) = \alpha \smallsmile_i \beta.
		\end{equation*}
		A \textbf{Cartan $i$-coboundary} is any map $\zeta_i$ in $\mathrm{Hom}(A \otimes A, A)$ that satisfies 
		\begin{equation*}
		(\partial \zeta_i)(\alpha \otimes  \beta) = 
		(\alpha \smallsmile_0 \beta) \smallsmile_i (\alpha \smallsmile_0 \beta)\ + \sum_{i=j+k} (\alpha \smallsmile_j \alpha) \smallsmile_0 (\beta \smallsmile_k \beta).
		\end{equation*}
	\end{definition}
	
	\begin{definition}
		The chain maps $F,\, G : \mathcal E(2) \rightarrow \mathcal E(4)$ are defined by
		\begin{equation*}
		\begin{split}
		F(\sigma_0,\dots,\sigma_n) =\ &  \bm{\circ}_{\mathcal E} \big( (\sigma_0,\dots,\sigma_n) \otimes \tilde{x}_0 \otimes \tilde{x}_0 \big) \\
		G(\sigma_0,\dots,\sigma_n) =\ &  \bm{\circ}_{\mathcal E} \big( \tilde{x}_0 \otimes AW (\sigma_0, \dots ,\sigma_n)^{\otimes 2} \big)
		\end{split}
		\end{equation*}
		for any basis element $(\sigma_0, \dots, \sigma_n)$.
	\end{definition}
	
	\begin{lemma} \label{lemma: value of (23)F and G on aabb}
		Let $A$ be an $\mathcal E$-algebra. For $\alpha, \beta \in A$ and $i \geq 0$ we have
		\begin{equation} \label{equation: (23)F(x_i) aabb = ... }
		(23) F (\tilde{x}_i) (\alpha \otimes \alpha \otimes \beta \otimes \beta ) = (\alpha \smallsmile_0 \beta) \smallsmile_i (\alpha \smallsmile_0 \beta)
		\end{equation}
		and
		\begin{equation} \label{equation: G(x_i) = ... }
		G(\tilde{x}_i) (\alpha \otimes \alpha \otimes \beta \otimes \beta) =	\sum_{i=j+k} (\alpha \smallsmile_j \alpha) \smallsmile_0 (\beta \smallsmile_k \beta).
		\end{equation}
	\end{lemma}
	
	\begin{proof}
		For (\ref{equation: (23)F(x_i) aabb = ... }) we compute
		\begin{equation*}
		\begin{split}
		(23) F (\tilde{x}_i)(\alpha \otimes \beta \otimes \alpha \otimes \beta ) &=  
		(23) \bm{\circ}_{\mathcal E} \big( \tilde{x}_i \otimes \tilde{x}_0 \otimes \tilde{x}_0 \big)(\alpha \otimes \alpha \otimes \beta \otimes \beta ) \\ &= 
		\bm{\circ}_{\mathcal E} \big( \tilde{x}_i \otimes \tilde{x}_0 \otimes \tilde{x}_0 \big)(\alpha \otimes \beta \otimes \alpha \otimes \beta ) \\ &= 
		\tilde{x}_i \big( \tilde{x}_0 (\alpha \otimes \beta) \otimes \tilde{x}_0 (\alpha \otimes \beta) \big) \\ &=
		(\alpha \smallsmile_0 \beta) \smallsmile_i (\alpha \smallsmile_0 \beta).
		\end{split}
		\end{equation*}
		For (\ref{equation: G(x_i) = ... }) we notice that
		\begin{equation*}
		AW( \tilde{x}_i \otimes \tilde{x}_i) = \sum_{j=0}^i \tilde{x}_j \otimes (12)^j \tilde{x}_{i-j}
		\end{equation*} 
		and compute
		\begin{equation*}
		\begin{split}
		G (\tilde{x}_i)( \alpha  \otimes \alpha \otimes \beta \otimes \beta ) &=  
		\sum_{j=0}^i \bm{\circ}_{\mathcal E} \big( \tilde{x}_0 \otimes  \tilde{x}_j \otimes (12)^j \tilde{x}_{i-j} \big)( \alpha  \otimes \alpha \otimes \beta \otimes \beta ) \\ &= 
		\sum_{j=0}^i \tilde{x}_0 \big( \tilde{x}_j (\alpha \otimes \alpha) \otimes \tilde{x}_{i-j} (\beta \otimes \beta) \big) \\ &=
		\!\!\!\sum_{i=j+k} (\alpha \smallsmile_j \alpha) \smallsmile_0 (\beta \smallsmile_k \beta)
		\end{split}
		\end{equation*}
		as desired.
	\end{proof}
	
	\begin{lemma} \label{lemma: conditions for a cartan coboundary}
		Let $H : \mathcal E(2) \rightarrow \mathcal E(4)$ be a chain homotopy between $(23) \circ F$ and $G$ satisfying
		\begin{equation*}
		H \circ (12) = (12)(34) \circ H.
		\end{equation*}
		Then, for any $\mathcal E$-algebra $A$ and $i \geq 0$, the map $\zeta_i: A \otimes A \to A$ defined by 
		\begin{equation*}
		\zeta_i(\alpha \otimes \beta) = H(\tilde{x}_i)(\alpha \otimes \alpha \otimes \beta \otimes \beta)
		\end{equation*}
		is a Cartan $i$-coboundary.
	\end{lemma}
	
	\begin{proof}
		By assumption
		\begin{equation*}
		\partial H = \partial \circ H + H \circ  \partial = (23) \circ F + G
		\end{equation*}
		so evaluating on $\tilde{x}_i$ we get
		\begin{equation} \label{pancito}
		\partial (H(\tilde{x}_i)) + H \big( (12)\tilde{x}_{i-1} + \tilde{x}_{i-1} \big) = (23) F(\tilde{x}_i) + G(\tilde{x}_i).
		\end{equation}
		Since $ \big((12)(34) + \mathrm{id}\big) (\alpha \otimes \alpha \otimes \beta \otimes \beta) = 0$, the equivariance assumption implies
		\begin{equation*}
		H \big( (12)\tilde{x}_{i-1} + \tilde{x}_{i-1} \big) (\alpha \otimes \alpha \otimes \beta \otimes \beta) = 0.
		\end{equation*}
		Evaluating (\ref{pancito}) on $(\alpha \otimes \alpha \otimes \beta \otimes \beta)$ and using Lemma \ref{lemma: value of (23)F and G on aabb} we obtain
		\begin{equation*}
		(\partial \zeta_i)(\alpha \otimes \beta) = 
		(\alpha \smallsmile_0 \beta) \smallsmile_i (\alpha \smallsmile_0 \beta) +
		\sum_{i=j+k} (\alpha \smallsmile_j \alpha) \smallsmile_0 (\beta \smallsmile_k \beta)
		\end{equation*}
		as desired.
	\end{proof}
	
	\subsection{Statement of the main theorem}
	
	We now describe the construction of a chain homotopy as stated in \mbox{Lemma \ref{lemma: conditions for a cartan coboundary}}.
	
	The first step of our construction is motivated by the well known fact that two group homomorphisms that are conjugate of each other induce homotopic maps of classifying spaces.
	
	Let us consider the three group inclusions
	\begin{equation*}
	i_1(\sigma)  = \sigma \times e \qquad
	i_2(\sigma) = e \times \sigma \qquad
	D(\sigma)   = \sigma \times \sigma.
	\end{equation*}
	of $\Sigma_2$ into $\Sigma_2 \times \Sigma_2$.
	
	\begin{definition}
		The homomorphisms $f, g : \Sigma_2 \to \Sigma_4$ are defined by
		\begin{align*}
		f &= \bm{\circ}_{\Sigma} \circ (\mathrm{id} \times D)\circ  i_1 \\
		g &= \bm{\circ}_{\Sigma} \circ (\mathrm{id} \times D)\circ  i_2
		\end{align*}
		or, more explicitly, by
		\begin{align} \label{equation: explicit values of f and g}
		f(e) &= e \qquad f(12) = (13)(24) \\ 
		g(e) &= e \qquad g(12) = (12)(34).
		\end{align} 
	\end{definition}
	
	We notice that the maps $f$ and $g$ are conjugated by $(23)$ in $\Sigma_4$, i.e., for $\sigma \in \Sigma_2$
	\begin{equation*}
	(23) f(\sigma) = g(\sigma) (23).
	\end{equation*}
	
	\begin{definition}
		Let $F,G : E(2) \to E(4)$ be the simplicial maps induced from $f$ and $g$ respectively. Explicitly,
		\begin{align*}
		& E(f)(\sigma_0, \dots, \sigma_n) = (f\sigma_0, \dots, f\sigma_n) \\
		& E(g)(\sigma_0, \dots, \sigma_n) = (g\sigma_0, \dots, g\sigma_n)
		\end{align*}
		for any $(\sigma_0, \dots, \sigma_n) \in E(2)$.
	\end{definition} 
	
	We introduce a linear map $H_1$ which, as we will see later, is an appropriately equivariant chain homotopy between $N_*(E(f))$ and $N_*(E(g))$.
	
	\begin{definition}
		Let $H_1 : \mathcal{E}(2) \to \mathcal{E}(4)$ be the degree $1$ linear map defined on basis elements by
		\begin{equation*}
		H_1 (\sigma_0, \dots, \sigma_n) = 
		\sum _{i=0}^n \big( (23) f \sigma_0, \dots,  (23) f \sigma_i, g \sigma_i, \dots, g \sigma_n \big).
		\end{equation*}
	\end{definition}
	
	The second step relates the maps $F$ and $N_*(E(f))$, as well as  $G$ and $N_*(E(g))$. It turns out, as we will see later, that the maps in the first pair are equal, whereas the maps in the second are chain homotopic via an appropriately equivariant map that we now introduce.
	
	\begin{definition}
		Let $H_2 : \mathcal E(2) \to \mathcal E(4)$ be the degree $1$ linear map defined on basis elements by
		\begin{equation*}
		H_2 (\sigma_0, \dots, \sigma_n) = N_*(\bm{\circ}_{E}) \big( (e, \dots, e) \otimes SHI (\sigma_0, \dots, \sigma_n)^{\otimes 2} \big).
		\end{equation*}
	\end{definition}
	
	We are ready to state our main result.
	\begin{theorem} \label{theorem: main}
		Let $A$ be an $\mathcal{E}$-algebra. For any $i \geq 0$ the map $\zeta_i$ in $\mathrm{Hom}(A \otimes A, A)$ defined by
		\begin{equation*}
		\zeta_i(\alpha \otimes \beta) = (H_1 + H_{2})(\tilde{x}_i)(\alpha \otimes \alpha \otimes \beta \otimes \beta)
		\end{equation*}
		is a Cartan $i$-coboundary.
	\end{theorem}

	\subsection{Proof of the main theorem}
	
	We will prove Theorem \ref{theorem: main} after four lemmas. 
	
	\begin{lemma} \label{lemma: boundary of H_1}
		In $\mathrm{Hom}\big( \mathcal E(2), \mathcal E(4) \big)$
		\begin{equation*}
		\partial  H_1 = (23) \circ N_*(E(f)) + N_*(E(g)).
		\end{equation*}
	\end{lemma}
	
	\begin{proof}
		The essence of the proof is a telescopic sum argument. 
		\begin{equation*}
		\begin{split}
		H_1 \partial (\sigma_0, \dots, \sigma_n) & = H_1 \sum_{j=0}^n (\sigma_0, \dots, \widehat{\sigma}_j, \dots, \sigma_n) \\ &= 
		\sum_{j=0}^n \sum_{i < j} ((23) f \sigma_0, \dots, (23) f \sigma_i, g \sigma_i, \dots, \widehat{\sigma}_j, \dots, g \sigma_n) \\ &+
		\sum_{j=0}^n \sum_{i > j} ((23) f \sigma_0, \dots, \widehat{\sigma}_j, \dots (23) f \sigma_i, g \sigma_i, \dots, g \sigma_n) \\ &=
		\partial H_1 (\sigma_0, \dots, \sigma_n) \\ &+
		\sum_i (d_i + d_{i+1}) ((23) f \sigma_0, \dots, (23) f \sigma_i, g \sigma_i, \dots, g \sigma_n) \\ &=
		\partial H_1 (\sigma_0, \dots, \sigma_n) \\ &+
		(23) N_*(E(f)) (\sigma_0, \dots, \sigma_n) + N_*(E(g)) (\sigma_0, \dots, \sigma_n)
		\end{split}
		\end{equation*}
		for any basis element $(\sigma_0,\dots,\sigma_n) \in \mathcal E(2)$.
	\end{proof}	
	
	\begin{lemma} \label{lemma: equivariance of H_1}
		The map $H_1$ satisfies the following form of equivariance:
		\begin{equation} \label{equation: lemma: equivariance of H_1}
		H_1 \circ (12) = (12)(34) \circ H_1.
		\end{equation}
	\end{lemma}
	
	\begin{proof}
		Let $H_1^i$ be the degree $1$ linear map defined on basis elements by
		\begin{equation*}
		H_1^i(\sigma_0, \dots, \sigma_n) = ((23) f \sigma_0, \dots, (23) f \sigma_i, g \sigma_i, \dots, g \sigma_n).
		\end{equation*}
		Notice that $H_1 = \sum_i H_1^i$. We will verify (\ref{equation: lemma: equivariance of H_1}) by checking it for each $H_1^i$. Using (\ref{equation: explicit values of f and g}) we notice that for any $\sigma$ in $\Sigma_2$ we have
		\begin{align*}
		(23) f \sigma &= (g \sigma) (23) \\
		g \big( (12)\sigma \big) &= (12)(34) g \sigma.
		\end{align*} 
		Therefore,
		\begin{equation*}
		\begin{split}
		(12)(34) H_1^i &(\sigma_0, \dots, \sigma_n)  \\ = & \  
		(12)(34) \big( (23) f \sigma_0, \dots, (23) f \sigma_i, g \sigma_i, \dots, g \sigma_n \big) \\ = & \
		(12)(34) \big( g \sigma_0 (23), \dots,  g \sigma_i (23), g \sigma_i, \dots, g \sigma_n \big) \\ = & \
		\big( g((12) \sigma_0) (23), \dots,  g((12) \sigma_i) (23), g((12) \sigma_i), \dots, g((12) \sigma_n) \big) \\ = & \
		\big( (23) f((12)\sigma_0), \dots, (23) f((12)\sigma_i), g((12)\sigma_i), \dots, g((12)\sigma_n) \big) \\ = & \
		H_1^i (12)(\sigma_0, \dots, \sigma_n)
		\end{split}
		\end{equation*}
		for any basis element $(\sigma_0,\dots,\sigma_n) \in \mathcal E(2)$.
	\end{proof}
	
	\begin{lemma} \label{lemma: boundary of H_2}
		In $\mathrm{Hom}\big( \mathcal E(2), \mathcal E(4) \big)$
		\begin{equation*}
		 N_*(E(f)) =  F
		\end{equation*}
		and
		\begin{equation*}
		\partial H_2 =  N_*(E(g)) + G.
		\end{equation*}
	\end{lemma}
	
	\begin{proof}
		We have
		\begin{equation*}
		\begin{split}
		F(\sigma_0,\dots,\sigma_n) =& \,
		\bm{\circ}_{\mathcal E}\big( (\sigma_0,\dots,\sigma_n) \otimes \tilde{x}_0 \otimes \tilde{x}_0 \big) \\ =& \,
		N_* (\bm{\circ}_{E}) \big( (\sigma_0,\dots,\sigma_n) \otimes (e, \dots, e) \otimes (e, \dots, e) \big) \\ =& \, ( f \sigma_0, \dots, f \sigma_n ) = N_*(E(f))(\sigma_0,\dots,\sigma_n)
		\end{split}
		\end{equation*}
		and
		\begin{equation*}
		\begin{split}
		G (\sigma_0,\dots,\sigma_n) =& \,
		\bm{\circ}_{\mathcal E} \big( \tilde{x}_0 \otimes AW (\sigma_0,\dots,\sigma_n)^{\otimes 2} \big) \\ =& \,
		N_* (\bm{\circ}_{E}) \big( (e, \dots, e) \otimes EZAW (\sigma_0,\dots,\sigma_n)^{\otimes 2} \big) \\ =& \, 
		N_* (\bm{\circ}_{E}) \big( (e, \dots, e) \otimes (\sigma_0,\dots,\sigma_n)^{\otimes 2} \big) \\ +& \,
		N_* (\bm{\circ}_{E}) \big( (e, \dots, e) \otimes (\partial SHI) (\sigma_0,\dots,\sigma_n)^{\otimes 2} \big) \\ =& \, 
		\big( g\sigma_0, \dots, g\sigma_n \big)\ +\ (\partial H_2)(\sigma_0,\dots,\sigma_n) \\ =& \,
		N_*(E(g))(\sigma_0,\dots,\sigma_n) \ +\ (\partial H_2)(\sigma_0,\dots,\sigma_n)
		\end{split}
		\end{equation*}
		for any basis element $(\sigma_0,\dots,\sigma_n) \in \mathcal E(2)$.
	\end{proof}
	
	\begin{lemma} \label{lemma: equivariance of H_2}
		The map $H_2$ satisfies the following form of equivariance:
		\begin{equation*}
		H_2 \circ (12) = (12)(34) \circ H_2.
		\end{equation*}
	\end{lemma}
	
	\begin{proof}
		Since for any $\sigma, \sigma' \in \Sigma_2$
		\begin{equation*}
		\bm{\circ}_{\bm{\Sigma}}\big( e, (12)\sigma, (12)\sigma' \big) = (12)(34) \bm{\circ}_{\bm{\Sigma}} ( e, \sigma, \sigma' )
		\end{equation*}
		we have
		\begin{equation*}
		\begin{split}
		H_2  (12) (\sigma_0, \dots, \sigma_n) = &\ 
		N_* (\bm{\circ}_{E}) \big( (e, \dots, e) \otimes SHI \big((12)\sigma_0, \dots, (12)\sigma_n \big)^{\otimes 2} \big) \\ = &\
		N_* (\bm{\circ}_{E}) \big( (e, \dots, e) \otimes \big((12) \otimes (12)\big) SHI \big(\sigma_0, \dots, \sigma_n \big)^{\otimes 2} \big) \\ = & \
		(12)(34)\, N_* (\bm{\circ}_{E}) \big( (e, \dots, e) \otimes SHI \big(\sigma_0, \dots, \sigma_n \big)^{\otimes 2} \big) \\ = & \
		(12)(34)\, H_2 (\sigma_0, \dots, \sigma_n)
		\end{split}
		\end{equation*}
		for any basis element $(\sigma_0, \dots, \sigma_n) \in \mathcal E(2)$.
	\end{proof}
	
	\begin{proof}[Proof of Theorem \ref{theorem: main}]
		We will show that $H = H_1 + H_2$ satisfies the hypothesis of Lemma \ref{lemma: conditions for a cartan coboundary}. Notice that the equivariance condition follows from Lemma \ref{lemma: equivariance of H_1} and Lemma \ref{lemma: equivariance of H_2}. To verify that $H$ is a chain homotopy between $(23) \circ F$ and $G$ we use Lemma \ref{lemma: boundary of H_1} and Lemma \ref{lemma: boundary of H_2}
		\begin{equation*}
		\begin{split}
		\partial H =\ & \partial H_1 + \partial H_2 \\ =\ & 
		(23) \circ N_*(E(f)) + N_*(E(g)) + N_*(E(g)) + G\\ =\ & 
		(23) \circ F + G
		\end{split}
		\end{equation*}
		as desired.
	\end{proof}

	\section{The $\mathcal E$-algebra structure on normalized cochains}
	
	In this section we effectively describe an $\mathcal E$-algebra structure on the normalized cochains of simplicial sets. This will allow us to apply our construction of Cartan $i$-coboundaries to this central example. Following \cite{berger04combinatorial}, we achive this in two steps corresponding to the two operad morphisms
	\begin{equation*}
	\mathcal{E} \to Sur \to \mathrm{End}(N^*(X)).
	\end{equation*}
	Here $Sur$ is the Surjection operad of McClure-Smith \cite{mcclure03cochain} and Berger-Fresse \cite{berger04combinatorial}, the first arrow is the Table Reduction morphism of \cite{berger04combinatorial}, and the second is the natural $Sur$-algebra structure introduced in \cite{mcclure03cochain} and \cite{berger04combinatorial}.
	
	\subsection{The Surjection operad} We review the definition of the Surjection operad of McClure-Smith \cite{mcclure03cochain} and Berger-Fresse \cite{berger04combinatorial}.
	
	For a non-negative integer $n$ define
	\begin{equation*}
	\overline{n} = \begin{cases}
	\{1, \dots, n\} & n > 0 \\
	\qquad \emptyset & n = 0.
	\end{cases}
	\end{equation*}
	Fix $r \geq 1$ and consider all $n \geq 0$. We make the free $\mathbb F_2$-module generated by all functions $s : \overline{n} \to \overline{r}$ into a chain complex by declaring the degree of $s : \overline{n} \to \overline{r}$ to be $r-n$ and its boundary to be
	\begin{equation*}
	\partial s = \sum_{k=1}^{n} s \circ \delta_k
	\end{equation*}
	where $\delta_k : \overline{n-1} \to \overline{n}$ is the injective order preserving function missing $k$. 
	
	We define $Sur(r)$ to be the quotient of this chain complex by the submodule generated by the functions $s : \overline{n} \to \overline{r}$ which are either non-surjective or for which $s(i)$ equals $s(i+1)$ for some $i$. 
	
	The chain complex $Sur(r)$ carries an action of $\Sigma_r = \{\overline{r} \to \overline{r}\ |\ \text{bijective}\}$ given by composition. The collection $Sur = \{Sur(r)\}_{r \geq 0}$ is an operad with partial composition $\circ_p: Sur(r') \otimes Sur(r) \to Sur(r+r'-1)$ defined on two generators $s : \overline{n} \to \overline{r}$ and $s' : \overline{n'}\to\overline{r'}$ as follows. Represent the surjections $s$ and $s'$ by sequences $(s(1), \dots, s(n))$ and $(s'(1), \dots, s'(n'))$ and suppose that $p$ appears $k$ times in the sequence representing $s'$ as $p = s'(i_1) = \dots = s'(i_k)$. Denote the set of all tuples
	\begin{equation*}
	1 = j_0 \leq j_1 \leq \dots \leq j_k = n
	\end{equation*}
	by $J(k,n)$ and for each such tuple consider the subsequences 
	\begin{equation*}
	(s(j_0), \dots, s(j_1)) \qquad (s(j_1),\dots,s(j_2)) \qquad \cdots \qquad (s(j_{k-1}), \dots, s(j_k)).
	\end{equation*}
	Then, in $(s'(1),\dots,s'(n))$, replace the value $s'(i_t)$ by $(s(j_{t-1}),\dots,s(j_t))$. In addition, increase the terms $s(j)$ by $r-1$ and the terms $s'(i)$ such that $s'(i)>r$ by $m-1$. The surjection $s'\circ_p s$ is represented by the sum, parametrized by $J(k,n)$, of these resulting sequences.
	
	\begin{example}
		Let us consider 
		\begin{equation*}
		s = (1,2,1) \qquad s' = (1,2,3,2,1).
		\end{equation*}
		Then, $s' \circ_2 s$ equals
		\begin{equation*}
		\big( 1,(2,3,2),4,(2),1 \big) + \big( 1,(2,3),4,(3,2),1 \big) + \big( 1,(2),4,(2,3,2),1 \big)
		\end{equation*}
		where the internal parenthesis are included solely for expository purposes.
	\end{example}

	\subsection{The Table Reduction morphism}
	
	We review the definition of the Table Reduction operad morphism
	\begin{equation*}
	TR : \mathcal E \to Sur
	\end{equation*}
	introduced in \cite{berger04combinatorial}. 
	
	Given a basis element $(\sigma_0, \dots, \sigma_n) \in \mathcal E(r)_n$ we define
	\begin{equation*}
	TR(\sigma_0, \dots, \sigma_n) = \sum_{a} s_{a}
	\end{equation*}
	as a sum of surjections
	\begin{equation*}
	s_{a} : \{1, \dots, n+r \} \to \{1, \dots, r\}
	\end{equation*}
	parametrized by all $a = (a_0, \dots, a_n)$ with each $a_i \geq 1$ and $a_0 + \cdots + a_n = n + r$.
	
	For one such tuple $a$ we now describe its associated surjection $s_a$. Consider the table
	\begin{center}
		\begin{tabular} {c | c c c}
			& $\sigma_0(1)$ & $\cdots$ & $\sigma_0(r)$ \\
			& $\sigma_1(1)$ & $\cdots$ & $\sigma_1(r)$ \\
			& $\vdots$      & $\ddots$ & $\vdots$      \\
			& $\sigma_n(1)$ & $\cdots$ & $\sigma_n(r)$
		\end{tabular}
	\end{center}
	define recursively
	\begin{equation*}
	A_{-1} = 0 \qquad A_i = A_{i-1} + a_{i.}
	\end{equation*}
	For $k \in \{1, \dots, n+r\}$ we identify $i \in \{1, \dots, n\}$ such that $A_{i-1} < k \leq A_{i}$ and define $s_a(k)$ to be the $(k - A_{i-1})$-th element in $(\sigma_i(1), \dots, \sigma_i(r))$ not in
	\begin{equation*}
	\big\{ s_a(j) \ | \ j < k \text{ and } j \neq A_0, \dots , A_{i-1} \big\}.
	\end{equation*}
	
	\begin{example}
		Let us compute $TR\big( (23), e, (12)(34) \big)$. Its associated table is
		\begin{center}
			\begin{tabular} {c | c}
				& 1 3 2 4 \\
				& 1 2 3 4 \\
				& 2 1 4 3
			\end{tabular}
		\end{center}
		and the tuples parametrizing the sum correspond to all permutations of $\{(1,1,4), (1,2,3), (2,2,2)\}$. Then,
		\begin{equation*}
		\begin{split}
		TR \big( (23), e, (12)(34) \big) = &\ (1, 3, 2, 3, 4, 3) \\ + &\ 
		(1, 1, 2, 3, 4, 4)+
		(1, 3, 2, 4, 4, 4)+
		(1, 1, 2, 1, 4, 3) \\ + &\
		(1, 3, 2, 2, 2, 4) +
		(1, 1, 2, 3, 4, 3)+
		(1, 3, 2, 2, 4, 4) \\ + &\
		(1, 3, 2, 3, 4, 4)+
		(1, 1, 2, 2, 4, 3)+
		(1, 3, 2, 2, 4, 3) \\ = &\
		(1, 3, 2, 3, 4, 3).
		\end{split}
		\end{equation*}
	\end{example}
	
	\subsection{The diagonal and join maps}
	
	We review the natural $Sur$-algebra structure on $N^*(X)$ introduced in \cite{mcclure03cochain} and \cite{berger04combinatorial} using the perspective presented in \cite{medina2018algebraic}. We notice that this \mbox{$Sur$-algebra} structure provides $N^*(X)$ with an $\mathcal E$-algebra structure
	\begin{equation*}
	\mathcal E \to Surj \to \mathrm{End}(N^*(X)). 
	\end{equation*}
	
	Let $\bm{\Delta}^n$ be the representable simplicial set $\mathrm{Hom}_{\bm{\Delta}}(-, [n])$. We identify a morphism $[m] \to [n]$ with its image $\{v_0, \dots, v_m\} \subseteq \{0, \dots, n\}$. As usual, it suffices to describe the $Surj$-structure for representable simplicial sets only.
	
	Let $\Delta: N^*(\bm{\Delta}^n) \to N^*(\bm{\Delta}^n)^{\otimes 2}$ be defined on basis elements by
	\begin{equation*}
	\Delta \{v_0, \dots, v_m\} = \sum_{i = 0}^{m} \{v_0, \dots, v_i\} \otimes \{v_i, \dots, v_m\}
	\end{equation*}
	and, for any $k \geq 2$, let $\ast : N^*(\bm{\Delta}^n)^{\otimes k} \to N^*(\bm{\Delta}^n)$ be defined on basis elements by 
	\begin{equation*}
	\ast(a_1 \otimes \cdots \otimes a_k) = \begin{cases}
	\bigcup_{i=1}^k a_i & \forall\ i < j,\ a_i \cap a_j = \emptyset \\
	\,0 & \exists\ i < j,\ a_i \cap a_j \neq \emptyset.
	\end{cases}
	\end{equation*}
	We notice that $\Delta$ is equal to the composition of the Alexander-Whitney map and the doubling map
	\begin{align*}
	N^*(\bm{\Delta}^n) &\to N^*(\bm{\Delta}^n \times \bm{\Delta}^n) \\
	a \quad &\mapsto \quad a \times a
	\end{align*}
	and that $\ast$ is commutative.
	
	We now describe the $Sur$-algebra structure on $N^*(\bm{\Delta}^n)$. Let
	\begin{equation*}
	s : \{1, \dots, n+r \} \to \{1, \dots, r\}
	\end{equation*}
	be a surjection, then $s(\alpha_1 \otimes \cdots \otimes \alpha_r)$ is defined by
	\begin{equation} \label{equation: defining lowercase phi}
	s(\alpha_1 \otimes \cdots \otimes \alpha_r)(a) = (\alpha_1 \otimes \cdots \otimes \alpha_r) \big(\ast_{s^{-1}(1)} \otimes \cdots \otimes \ast_{s^{-1}(r)} \big)\ \Delta^{r+d-1} (a)
	\end{equation}
	where $\Delta^{k}$ is recursively defined by
	\begin{equation*}
	\begin{split}
	&\Delta^1 = \Delta \\ &\Delta^{k+1} = (\Delta \otimes \mathrm{id}^{\otimes k})\, \Delta^k
	\end{split}
	\end{equation*} 
	and $\ast_{s^{-1}(i)}$ is given by applying $\ast$ to the factors in positions $s^{-1}(i)$. 
	
	\begin{remark}
		This $Surj$-algebra structure on normalized cochains of simplicial sets is induced, as explained in \cite{medina2018cellular}, from an explicit cellular $E_\infty$-bialgebra structure on the geometric realization of $\bm{\Delta^1}$.
	\end{remark}
	
	\appendix
	
	\section{Explicit examples}
	
	Let $\bm{\Delta}_n$ be the $n$-th representable simplicial set and denote $\{0, \dots, n\} \in N_n(\bm{\Delta}^n)$ by $\mathrm{id}_n$. In this appendix, we will compute for an arbitrary pair of homogeneous cocycles $\alpha, \beta \in N^*(\bm{\Delta}_n)$ the value
	\begin{equation*}
	\zeta_i(\alpha \otimes \beta)(\mathrm{id}_n)
	\end{equation*}
	in terms of $\alpha$ and $\beta$. We will restrict to $i = 0,1,2$ and $n$ the smallest integer for which this value is not identically $0$. 
	
	Recall that 
	\begin{equation*}
	\tilde{x}_i = (e, (12), e, \dots, (12)^i) \in \mathcal E(2)
	\end{equation*} 
	and that our Cartan $i$-coboundary is defined by 
	\begin{equation*}
	\zeta_i(\alpha \otimes \beta) = TR(H_1 + H_2)(\tilde{x}_i)\,(\alpha \otimes \alpha \otimes \beta \otimes \beta)
	\end{equation*}
	where $TR : \mathcal E \to Surj$ is the Table Reduction morphism and
	\begin{equation*}
	H_1 (\sigma_0, \dots, \sigma_n) = 
	\sum _{i=0}^n \big( (23) f \sigma_0, \dots,  (23) f \sigma_i, g \sigma_i, \dots, g \sigma_n \big)
	\end{equation*}
	\begin{equation*}
	H_2 (\sigma_0, \dots, \sigma_n) = N_*(\bm{\circ}_{E}) \big( (e, \dots, e) \otimes SHI (\sigma_0, \dots, \sigma_n)^{\otimes 2} \big).
	\end{equation*}
	
	We remark that the cochain $\zeta_i(\alpha \otimes \beta)$ witnesses the relation
	\begin{equation*}
	Sq^{p+q-i} \big( [\alpha]\, [\beta] \big) = \sum_{p+q-i = j+k} Sq^j([\alpha])\, Sq^k([\beta])
	\end{equation*}
	where $-p$ and $-q$ are the degrees of $\alpha$ and $\beta$.
	
	\begin{example} \label{example: Cartan coboundary 0}
		For $i=0$, we have
		\begin{equation*}
		H_1(\tilde{x}_0) = ((23)fe, ge) = ((23), e)
		\end{equation*}
		and
		\begin{equation*}
		H_2(\tilde{x}_0) = 0.
		\end{equation*}
		Therefore,
		\begin{equation*}
		TR \big( (H_1 + H_2) (\tilde{x}_0) \big) = (1,3,2,3,4)
		\end{equation*}
		and 
		\begin{equation*}
		\zeta_0(\alpha \otimes \beta)\{0,1,2,3\} = \alpha \{0, 1\} \, \alpha\{1, 2\} \, \beta \{1, 2\} \, \beta \{2, 3\}.
		\end{equation*}
	\end{example}
	
	\begin{example} \label{example: Cartan coboundary 1}
		For $i=1$, we have
		\begin{equation*}
		H_1(\tilde{x}_1) = 
		\big( (23), e, (12)(34) \big) + \big( (23), (23)(13)(24), (12)(34) \big)
		\end{equation*}
		and
		\begin{equation*}
		H_2 (\tilde{x}_1) =
		\big( e, (34), (12)(34) \big).
		\end{equation*}
		Therefore,
		\begin{equation*}
		TR \big( (H_1 + H_2)(\tilde{x}_1) \big) =
		(1,3,2,3,4,3)\ +\ (1,2,4,1,4,3)
		\end{equation*}
		and 
		\begin{equation*}
		\begin{split}
		\zeta_1(\alpha \otimes \beta)\{0,1,2,3,4\} =\ &
		\alpha\{0, 1\} \, \alpha\{1, 2\} \, \beta\{1, 2, 4\} \, \beta\{2, 3, 4\}\\ +\ 
		& \alpha\{0, 2, 3\} \, \alpha\{0, 1, 2\} \, \beta\{3, 4\} \, \beta\{2, 3\}.
		\end{split}
		\end{equation*}
	\end{example}
	
	\begin{example} \label{example: Cartan coboundary 2}
		For $i=2$, we have
		\begin{equation*}
		\begin{split}
		H_1(\tilde{x}_2) = &\ 
		\big( (23), e, (12)(34), e \big) +
		\big( (23), (12)(34)(23), (12)(34), e \big) \\ + &\ 
		\big( (23), (12)(34)(23), (23), e \big)
		\end{split}
		\end{equation*}
		and
		\begin{equation*}
		\begin{split}
		H_2(\tilde{x}_2) = &\ 
		\big( e, (34), (12)(34), (12) \big) +
		\big( e, e, (12), e \big)  \\ + &\
		\big( e, (34), e, (12) \big) +
		\big( e, (12)(34), (12), e \big).
		\end{split}
		\end{equation*}
		Therefore,
		\begin{equation*}
		\begin{split}
		TR \big( (H_1 + H_2)(\tilde{x}_2) \big) =\ & 
		(1, 3, 2, 3, 4, 3, 4) + (1, 2, 4, 1, 4, 3, 4) \\ +\ & 
		(1, 2, 1, 3, 2, 3, 4) + (1, 3, 2, 3, 2, 3, 4)
		\end{split}
		\end{equation*}
		and
		\begin{equation*}
		\begin{split}
		\zeta_2(\alpha \otimes \beta)\{0,1,2,3,4,5\} & =\ 
		\alpha\{0, 1\}\, \alpha\{1, 2\}\, \beta\{1, 2, 3, 4\}\, \beta\{2, 3, 4, 5\} \\ & +\ 
		\alpha\{0, 1\}\, \alpha\{1, 2\}\, \beta\{1, 2, 4, 5\}\, \beta\{2, 3, 4, 5\} \\ & +\ 
		\alpha\{0, 2, 3\}\, \alpha\{0, 1, 2\}\, \beta\{3, 4, 5\}\, \beta\{2, 3, 5\} \\ & +\
		\alpha\{0, 1, 2, 3\}\, \alpha\{0, 1, 3, 4\}\, \beta\{3, 4\}\, \beta\{4, 5\} \\ & +\
		\alpha\{0, 1, 2, 3\}\, \alpha\{1, 2, 3, 4\}\, \beta\{3, 4\}\, \beta\{4, 5\}.
		\end{split}
		\end{equation*}
	\end{example}
	
	\bibliographystyle{alpha} 
	\bibliography{chainlevelcartan}

\begin{thebibliography}{MM18d}

\bibitem[BF04]{berger04combinatorial}
Clemens Berger and Benoit Fresse.
\newblock Combinatorial operad actions on cochains.
\newblock In {\em Mathematical Proceedings of the Cambridge Philosophical
  Society}, volume 137, pages 135--174. Cambridge University Press, 2004.

\bibitem[BM16]{brumfiel2016pontrjagin}
Greg Brumfiel and John Morgan.
\newblock The pontrjagin dual of 3-dimensional spin bordism.
\newblock {\em arXiv preprint arXiv:1612.02860}, 2016.

\bibitem[BM18]{brumfiel2018quadratic}
Greg Brumfiel and John Morgan.
\newblock Quadratic functions of cocycles and pin structures.
\newblock {\em arXiv preprint arXiv:1808.10484}, 2018.

\bibitem[EML53]{eilenberg1953groups}
Samuel Eilenberg and Saunders Mac~Lane.
\newblock On the groups ${H}(\pi,n)$, {I}.
\newblock {\em Ann. of Math.(2)}, 58(1):55--106, 1953.

\bibitem[KT17]{kapustin2017fermionic}
Anton Kapustin and Ryan Thorngren.
\newblock Fermionic spt phases in higher dimensions and bosonization.
\newblock {\em Journal of High Energy Physics}, 2017(10):80, 2017.

\bibitem[LV12]{loday2012operads}
Jean-Louis Loday and Bruno Vallette.
\newblock {\em Algebraic operads}, volume 346 of {\em Grundlehren der
  Mathematischen Wissenschaften [Fundamental Principles of Mathematical
  Sciences]}.
\newblock Springer, Heidelberg, 2012.

\bibitem[May70]{may70generalapproach}
J~Peter May.
\newblock A general algebraic approach to {S}teenrod operations.
\newblock In {\em The Steenrod Algebra and its Applications: a conference to
  celebrate NE Steenrod's sixtieth birthday}, pages 153--231. Springer, 1970.

\bibitem[MM18a]{medina2018axiomatic}
Anibal~M. Medina-Mardones.
\newblock An axiomatic characterization of {S}teenrod's cup-$i$ products.
\newblock {\em arXiv preprint arXiv:1810.06505}, 2018.

\bibitem[MM18b]{medina2018algebraic}
Anibal.~M. Medina-Mardones.
\newblock A finitely presented ${E}_{\infty}$-prop {I}: algebraic context.
\newblock {\em arXiv preprint arXiv:1808.00854}, 2018.

\bibitem[MM18c]{medina2018cellular}
Anibal~M. Medina-Mardones.
\newblock A finitely presented ${E}_{\infty}$-prop {II}: cellular context.
\newblock {\em arXiv preprint arXiv:1808.07132}, 2018.

\bibitem[MM18d]{medina2018persistence}
Anibal~M. Medina-Mardones.
\newblock Persistence {S}teenrod modules.
\newblock {\em arXiv preprint arXiv:1812.05031}, 2018.

\bibitem[MM19]{medina2019globular}
Anibal~M. Medina-Mardones.
\newblock An algebraic representation of globular sets.
\newblock {\em arXiv preprint arXiv:1906.01011}, 2019.

\bibitem[MS03]{mcclure03cochain}
James McClure and Jeffrey Smith.
\newblock Multivariable cochain operations and little n-cubes.
\newblock {\em Journal of the American Mathematical Society}, 16(3):681--704,
  2003.

\bibitem[Rea00]{real2000homological}
Pedro Real.
\newblock Homological perturbation theory and associativity.
\newblock {\em Homology, Homotopy and Applications}, 2(1):51--88, 2000.

\bibitem[Rub91]{rubio1991homologie}
J~Rubio.
\newblock Homologie effective des espaces de lacets it{\'e}r{\'e}s: un
  logiciel.
\newblock {\em These de doctorat de l’Institut Fourier, Grenoble}, 1991.

\bibitem[SE62]{steenrod62operations}
Norman~Earl Steenrod and David~BA Epstein.
\newblock {\em Cohomology operations}.
\newblock Number 50-51. Princeton University Press, 1962.

\bibitem[Ste47]{steenrod47products}
Norman~E Steenrod.
\newblock Products of cocycles and extensions of mappings.
\newblock {\em Annals of Mathematics}, pages 290--320, 1947.

\end{thebibliography}
\end{document}